\documentclass[11pt]{article}

\usepackage{oldlfont}
\usepackage{amsfonts}
\usepackage{graphicx}
\usepackage{eepic}

\newcommand{\be}{\begin{equation}}
\newcommand{\ee}{\end{equation}}
\newcommand{\bea}{\begin{eqnarray}}
\newcommand{\eea}{\end{eqnarray}}
\newcommand{\beas}{\begin{eqnarray*}}
\newcommand{\eeas}{\end{eqnarray*}}
\newcommand{\ba}{\begin{array}}
\newcommand{\ea}{\end{array}}

\newcommand{\field}[1]{\mathbb{#1}}
\newcommand{\supp}{{\mathrm supp}}

\newcommand{\grad}{\nabla}
\renewcommand{\div}{{\mathrm div}}
\newcommand{\curl}{{\mathrm curl}}

\newcommand{\g}{\gamma}
\newcommand{\eps}{\varepsilon}

\newcommand{\bolda}{{\mathbf a}}

\newcommand{\bn}{{\mathbf n}}

\newcommand{\bu}{{\mathbf u}}
\newcommand{\bv}{{\mathbf v}}

\newcommand{\bx}{{\mathbf x}}

\newcommand{\bH}{{\mathbf H}}

\newcommand{\bcurl}{{\mathbf curl}}

\newcommand{\bphi}{\hbox{\mathversion{bold}$\phi$}}

\newcommand{\CE}{{\cal E}}

\newcommand{\CP}{{\cal P}}
\newcommand{\bCP}{\hbox{\mathversion{bold}$\cal P$}}

\newcommand{\tH}{\tilde H}

\newtheorem{theorem}{Theorem} [section]
\newtheorem{lemma}{Lemma} [section]

\newtheorem{remark}{Remark} [section]
\newenvironment{proof}{\noindent\textbf{Proof.}\ }
              {\nopagebreak\hbox{ }\hfill$\Box$\bigskip}

\title{Optimal error estimation for $\bH(\curl)$-conforming
       $p$-interpolation in two dimensions
\thanks{Supported by EPSRC under grant no. EP/E058094/1.}}

\author{Alexei Bespalov
\thanks{Department of Mathematical Sciences, Brunel University,
        Uxbridge, West London UB8 3PH, UK.
        Email: {\tt albespalov@yahoo.com}}
        \and
        Norbert Heuer
\thanks{Facultad de Matem\'aticas, Pontificia Universidad Cat\'olica de Chile,
        Avenida Vicu\~na Mackenna 4860, Santiago, Chile.
        Email: {\tt nheuer@mat.puc.cl}}
        }
\begin{document}
\date{}
\maketitle

\begin{abstract}
In this paper we prove an optimal error estimate for 
the $\bH(\curl)$-conforming projection based $p$-interpolation operator
introduced in [L.~Demkowicz and I.~{Babu\v ska}, {\em $p$ interpolation error estimates
for edge finite elements of variable order in two dimensions}, SIAM J. Numer. Anal.,
{\bf 41} (2003), pp.~1195--1208]. This result is proved on the reference element
(either triangle or square) $K$ for regular vector fields in $\bH^r(\curl,K)$
with arbitrary $r >0$. The formulation of the result in the $\bH(\div)$-conforming
setting, which is relevant for the analysis of high-order boundary element
approximations for Maxwell's equations, is provided as well.

\bigskip
\noindent
{\em Key words}: finite element method, $p$-interpolation, error estimation,
                 Maxwell's equations

\noindent
{\em AMS Subject Classification}: 65N30, 65N15, 41A10
\end{abstract}

\section{Introduction} \label{sec_intro}
\setcounter{equation}{0}

This paper concerns the $\bH(\curl)$-conforming interpolation of regular vector
fields by high order polynomials on the reference triangle or square
and the corresponding interpolation error estimation.
To the best of our knowledge, the first paper related to this subject is the
paper by Suri from 1990 \cite{Suri_90_SCH}. In that paper error estimates
(in terms of both the mesh parameter $h$ and the polynomial degree $p$)
were derived for classical Raviart-Thomas (RT) and Brezzi-Douglas-Marini (BDM)
interpolation operators on the reference square $Q$ (note that in 2D these
$\bH(\div)$-conforming so-called face elements and the $\bH(\curl)$-conforming
edge elements of the N{\' e}d{\' e}lec type are isomorphic). The estimates obtained
were not optimal with respect to $p$ and later, in \cite{StenbergS_97_hpE},
they were improved to $\eps$-suboptimal $p$-estimates for sufficiently regular
vector fields $\bu$ (namely, for $\bu \in \bH^r(\div,Q)$ with $r > \frac 12$).
These results were further extended by Ainsworth and Pinchedez in
\cite{AinsworthP_02_hpA} (to meshes with hanging nodes, weighted Sobolev regularity
of approximated functions, exponential convergence on graded meshes) and in
\cite{AinsworthP_02_hpM} (to Brezzi-Douglas-Fortin-Marini elements with non-uniform
distribution of polynomial degrees). In 3D the corresponding results were obtained
by Monk in \cite{Monk_94_php} and by Ben Belgacem and Bernardi in \cite{BenBelgacemB_99_SED}.
All mentioned papers deal with quadrilateral or hexahedral elements and the proofs
therein essentially rely on expansions in terms of orthogonal (namely, Legendre)
polynomials. An application of this approach to triangular or tetrahedral elements
does not seem to be feasible. Another drawback of classical interpolation
operators for edge (or face) elements is the lack of stability (with respect to $p$)
for low-regular fields.

A breakthrough in $\bH(\curl)$-conforming $p$-interpolation analysis, i.e.,
the construction of an interpolation operator which works
equally well on both triangular and quadrilateral elements and also for
low-regular fields, was achieved relatively recently by Demkowicz and Babu{\v s}ka
in~\cite{DemkowiczB_03_pIE}. These authors have introduced and analyzed
$H^1$- and $\bH(\curl)$-conforming projection-based $p$-interpolation
operators satisfying the commuting diagram property (de Rham diagram).
This property and the corresponding $p$-interpolation error estimates have
immediate applications to the analysis of high-order finite element (FE) discretizations
of the time-harmonic Maxwell's equations. In particular, they are critical to
prove the discrete compactness property (which implies the convergence of FE
approximations for Maxwell's equations) and also useful for the error analysis
(see \cite{BoffiDC_03_DCp,BoffiCDD_06_Dhp}).
Moreover, the interpolation operators in~\cite{DemkowiczB_03_pIE} were constructed
to allow polynomial degrees to vary from one element to another. This has been done
by assigning to each element an ``internal'' polynomial degree and a sequence
of (possibly lower) ``edge'' degrees. Such a construction of interpolation operators
is essential for the analysis of exponentially convergent $hp$-approximations
and $hp$-adaptive schemes. In~\cite{DemkowiczB_05_HHH} these results have been
extended to the 3D case.

The error estimates presented in~\cite{DemkowiczB_03_pIE} for both $p$-interpolation
operators are suboptimal. Furthermore,
the error estimate of the $\bH(\curl)$-conforming interpolation is available only
for low-regular vector fields in $\bH^r(\curl,K)$ with $r \in (0,1)$
(here $K$ is either the reference triangle or square).
Though these drawbacks are not essential for the convergence analysis of FE
approximations, the corresponding improvements would be advantageous for the error
analysis. In this paper we show that an optimal estimate for the error
of the $H^1$-conforming $p$-interpolation can be obtained in the $H^1$-semi-norm.
Using this result we then prove an optimal error estimate for the
$\bH(\curl)$-conforming $p$-interpolation operator applied to a vector
field $\bu \in \bH^r(\curl,K)$ with arbitrary $r > 0$. In the proof we rely
on a regular splitting of $\bu \in \bH^r(\curl,K)$ into a $\curl$-free
component and a complementary vector field of extra smoothness. This splitting
is possible to recent results related to the regularized Poincar{\' e}-type integral
operators in~\cite{CostabelM_BRP}.

The paper is organized as follows. In the next section we introduce necessary notation
and formulate some auxiliary results.
In Section~\ref{sec_int} we quote \cite{DemkowiczB_03_pIE} to
briefly sketch the definition and properties of the $H^1$- and $\bH(\curl)$-conforming
projection-based $p$-interpolation operators. Optimal error estimates for both
interpolation operators are proved in Section~\ref{sec_estimate}
(see Theorems~\ref{thm_H1_estimate} and~\ref{thm_H(curl)_estimate}).
The paper is concluded with Section~\ref{sec_remarks} where we mention some simple
extensions of our results including the $H(\div)$-conforming $p$-interpolation and
$hp$-estimates.

Throughout the paper, $C$ denotes a generic positive constant which is independent
of $p$ and involved functions.

\section{Notation and auxiliary results} \label{sec_aux}
\setcounter{equation}{0}

We will present all technical details only for the equilateral reference triangle
$T = \{x_2 > 0,\ x_2 < \sqrt 3 (x_1 + 1),\ x_2 < (1 - x_1)\sqrt 3\}$. The case of the
reference square $Q = (-1,1)^2$ (for which the arguments are essentially the same)
is briefly discussed in Section~\ref{sec_remarks}. A generic side of the triangle $T$
will be denoted by $\ell$.

We will use the standard definitions for the Sobolev spaces $H^r$ ($r \ge 0$)
of scalar functions on the interval $I = (-1,1)$ and on the triangle $T$
(see, e.g., \cite{LionsMagenes}). The norms in these spaces are denoted by
$\|\cdot\|_{H^r(I)}$ and $\|\cdot\|_{H^r(T)}$, respectively.
On the interval $I$ we will also need the Sobolev spaces $\tH^r(I)$
for $r \in (0,1)$ which are defined by interpolation.
We use the real K-method of interpolation (see \cite{LionsMagenes}) to define
\[
   \tilde H^{r}(I) = \Big(L^2(I), H_0^t(I)\Big)_{\frac rt,2}
   \quad (1/2 < t \le 1,\ 0<r<t).
\]
Here, $H_0^t(I)$  ($0<t\le 1$) is the completion of $C_0^\infty(I)$ in
$H^t(I)$ and we identify $H_0^1(I)$ and $\tilde H^1(I)$.
Note that the Sobolev spaces $H^r$ also satisfy the interpolation property, e.g.,
\[
   H^r(I) = \Big(L^2(I), H^1(I)\Big)_{r,2}\quad (0<r<1)
\]
with equivalent norms. Furthermore, the Sobolev spaces $H^r$ and $\tH^r$
on any edge $\ell \subset \partial T$ are defined by using the
definitions of the corresponding spaces on the interval $I$.

Throughout the paper, we use boldface symbols for vector fields.
The spaces (or sets) of vector fields are denoted in boldface as well
(e.g., $\bH^r(T) = (H^r(T))^2$), with their norms and inner
products being defined component-wise.
The standard notation will be used also for differential operators
$\grad = (\partial/\partial x_1,\,\partial/\partial x_2)$,
$\div = \grad\,\cdot$, $\curl = \grad\mbox{\small $\times$}$, and
for the Laplace operator $\Delta = \div\,\grad$.

The $L^2$-inner product and the corresponding $L^2$-norm on $T$ are
denoted by $(\cdot,\cdot)_{0,T}$ and $\|\cdot\|_{0,T}$, respectively.
We will use the semi-norm in $H^1(T)$ which is defined as
\[
  |u|_{H^1(T)} = \|\grad\,u\|_{0,T}.
\]
Furthermore, we will use the space
\[
  \bH^r(\curl,T) :=
  \{\bu \in \bH^{r}(T);\;
  \curl\,\bu \in H^{r}(T)\},\quad r \ge 0,
\]
and its analog $\bH^r(\div,T)$ in the $\div$-setting. In both cases the spaces are
equipped with their graph norms. For $r = 0$ we drop the superscript in the above
notations: $\bH^0(\curl,T) = \bH(\curl,T)$ and $\bH^0(\div,T) = \bH(\div,T)$.

We will also need the space $H^{1/2}(\partial T)$ which
can be defined as the trace space of $H^1(T)$ on $\partial T$ with norm
\[
  \|u\|_{H^{1/2}(\partial T)} =
  \inf_{U|_{\partial T} = u} \|U\|_{H^1(T)}.
\]

Let us introduce the needed polynomial sets.
By $\CP_p(I)$ we denote the set of polynomials of degree $\le p$ on
the interval $I$, and $\CP^0_p(I)$ denotes the subset of $\CP_p(I)$
which consists of polynomials vanishing at the end points of $I$.
In particular, these two sets will be used for the edges $\ell \subset \partial T$.

Further, $\CP_p(T)$ is the set of polynomials on $T$ of total degree $\le p$.
The corresponding set of polynomial (scalar) bubble functions on $T$ is denoted
by $\CP^0_{p}(T)$. When considering the reference square $Q$ we will
denote by $\CP_{p}(Q)$ the set of polynomials of degree $\le p$ in each variable
separately. Finally, $\bCP_p^{\rm Ned}(T)$ denotes the polynomial set
associated with the second N{\' e}d{\' e}lec family of edge elements on $T$, i.e.,
$\bCP_p^{\rm Ned}(T) = (\CP_p(T))^2$ (see \cite{Nedelec_86_NFM}).
The subset of $\bCP^{\rm Ned}_p(T)$
which consists of vector-valued polynomials with vanishing tangential trace
on the boundary $\partial T$ (vector bubble-functions) will be denoted
by $\bCP^{\rm Ned,0}_p(T)$.

To simplify the presentation we have assigned only one (``internal'') polynomial degree
to the reference element. Following \cite{DemkowiczB_03_pIE} the results
extend to the polynomial spaces on $T$ with separate polynomial degrees
(of possibly lower order) assigned also to traces on the edges of $\partial T$.

When proving interpolation error estimates in Section~\ref{sec_estimate}
we will need some auxiliary results, which are collected in the next
three sub-sections.

\subsection{Polynomial extensions from the boundary} \label{sec_ext}

The problem of polynomial extension from the
boundary $\partial T$ can be formulated as follows: given a continuous function
$f$ defined on $\partial T$ such that $f|_{\ell} \in \CP_p(\ell)$
for any $\ell \subset \partial T$,
find a polynomial $\Phi \in \CP_p(T)$ satisfying $\Phi|_{\partial T} = f$.
The existence of such an extension, which is stable (with respect to $p$) as
a mapping $H^{1/2}(\partial T) \rightarrow H^1(T)$, has been proved in
\cite[Theorem~7.4]{BabuskaCMP_91_EPp}. In general (i.e., for $p > 2$),
the extension is not uniquely defined. To ensure the uniqueness, we will search
for discrete harmonic extensions:
given a continuous piecewise polynomial $f$ of degree $p$ on each side
$\ell \subset \partial T$, find a polynomial $F := \CE_p f \in \CP_p(T)$
such that $F = f$ on $\partial T$ and
\be \label{ext_1}
    (\grad\, F,\grad\,\varphi)_{0,T} = 0\qquad
    \forall\,\varphi \in \CP_p^0(T).
\ee
Then for any polynomial extension $\Phi \in \CP_p(T)$ such that $\Phi|_{\partial T} = f$
there holds
\[
  |\Phi|^2_{H^1(T)} = |(\Phi - \CE_p f) + \CE_p f|^2_{H^1(T)} =
  |\Phi - \CE_p f|^2_{H^1(T)} + |\CE_p f|^2_{H^1(T)} \ge
  |\CE_p f|^2_{H^1(T)}.
\]
Hence, applying the mentioned result of \cite{BabuskaCMP_91_EPp}, we have
\be \label{ext_2}
    |\CE_p f|_{H^1(T)} \le C\, \|f\|_{H^{1/2}(\partial T)}
\ee
with a positive constant $C$ independent of $p$ and $f$.

\begin{remark} \label{rem_ext}
{\rm (i)} The same result as above holds for the reference square $Q$
(see {\rm \cite[Theorem~7.5]{BabuskaCMP_91_EPp}}).

{\rm (ii)} If $f$ is a continuous piecewise linear function on $\partial T$ and
$F := \CE_1 f \in \CP_1(T)$, then $F = \CE_p f$ for any $p > 1$.
Indeed, $F \in \CP_1(T) \subset \CP_p(T)$, $F = f$ on $\partial T$, and
for every $\varphi \in \CP_p^0(T)$ there holds
\[
  (\grad\, F,\grad\,\varphi)_{0,T} = - (\Delta\,F, \varphi)_{0,T} = 0.
\]
\end{remark}

\subsection{Polynomial approximation of scalar functions} \label{sec_approx}

In this sub-section we provide two $p$-approximation results
for scalar functions with Sobolev regularity in one and two dimensions.
The first result is a careful revision of \cite[Theorem~3.2]{BabuskaS_89_TND},
where an additional ${\log^{1/2}}{p}\,$-term in the error bound appears.

\begin{lemma} \label{lm_1D_p-approx}
Let $I = (-1,1)$ and $f \in H^r(I)$ with some $r > \frac 12$.
Then there exists a polynomial $f_p \in \CP_p(I)$ such that
$f_p(\pm 1) = f(\pm 1)$ and
\be \label{1D_estimate}
    \|f - f_p\|_{\tH^{1/2}(I)} \le
    C\, p^{-(r-1/2)}\, \|f\|_{H^r(I)}.
\ee
\end{lemma}

\begin{proof}
We outline a proof following the ideas of~\cite{BabuskaS_87_OCR,BabuskaS_89_TND}.
Let $f \in H^r(I)$, $r > 1/2$. Then $f$ can be expanded as
\[
  f(x) = \sum\limits_{i=0}^{\infty} a_i T_i(x),
\]
where $T_i(x)$ is the Chebyshev polynomial of degree $i$ on $I$.

Setting
\[
  P_p f(x) := \sum\limits_{i=0}^{p} a_i T_i(x) \in \CP_p(I),
\]
one has for any $s \in (\frac 12, \min\,\{1,r\})$
(see \cite[Theorem~3.2,~Remark~3.2]{BabuskaS_89_TND})
\be \label{1D_1}
    \|f - P_p f\|_{H^s(I)} \le
    C\,p^{-(r-s)}\,\|f\|_{H^r(I)}.
\ee
Moreover, it is easy to show that (\ref{1D_1}) holds for $s = 0$:
\beas
     \|f - P_p f\|^2_{L^2(I)}
     & \le &
     \|(f - P_p f)(\cos\xi)\|^2_{L^2(0,\pi)} \simeq
     \sum\limits_{i=p+1}^{\infty} |a_i|^2
     \\
     & \le &
     C p^{-2r} \sum\limits_{i=p+1}^{\infty} |a_i|^2 (1 + i^2)^r \le
     C p^{-2r} \|f\|_{H^r(I)}^2.
\eeas
Analogously, applying the Schwarz inequality we prove that
\bea
    |(f - P_p f)(\pm 1)|^2
    & \le &
    \bigg(\sum\limits_{i=p+1}^{\infty} |a_i|\bigg)^2
    \nonumber
    \\[3pt]
    & \le &
    \sum\limits_{i=p+1}^{\infty} (1 + i^2)^{-r}\,
    \sum\limits_{i=p+1}^{\infty} |a_i|^2 (1 + i^2)^{r}
    \nonumber
    \\[5pt]
    & \le &
    C p^{-2(r-1/2)} \|f\|_{H^r(I)}^2.
    \label{1D_2}
\eea
Now we will adjust the polynomial $P_p f$ at the end points of $I$.
First, using the same idea as in the two-dimensional case
(see \cite[pp.~759--760]{BabuskaS_87_OCR}), we find two polynomials
$\psi_p^{+},\;\psi_p^{-} \in \CP_p(I)$ such that
\[
  \psi_p^{-} (- 1) = 1,\quad \psi_p^{-} (1) = 0,\qquad
  \psi_p^{+} (- 1) = 0,\quad \psi_p^{+} (1) = 1,
\]
and
\be \label{1D_3}
    \|\psi_p^{\pm}\|_{H^s(I)} \le C p^{s-1/2},\quad s \in (0,1).
\ee
Then we set
\[
  f_p(x) := P_p f(x) + (f - P_p f)(-1)\,\psi_p^{-}(x) + (f - P_p f)(1)\,\psi_p^{+}(x),
  \quad x \in I.
\]
It is easy to check that $f_p \in \CP_p(I)$ and $f_p(\pm 1) = f(\pm 1)$.
Furthermore, making use of (\ref{1D_1})--(\ref{1D_3}), we obtain for any
$s \in \{0\} \cup (\frac 12, \min\,\{1,r\})$
\beas
     \|f - f_p\|_{\tH^s(I)}
     & \simeq &
     \|f - f_p\|_{H^s(I)}
     \\
     & \le &
     \|f - P_p f\|_{H^s(I)} +
     \max_{x = \pm 1}|(f - P_p f)(x)|\,
     \Big(\|\psi_p^{-}\|_{H^s(I)} + \|\psi_p^{+}\|_{H^s(I)}\Big)
     \\
     & \le &
     C p^{-(r-s)} \|f\|_{H^r(I)}.
\eeas
The inequality in (\ref{1D_estimate}) now follows via interpolation
between $H^0(I)$ and $\tH^{s}(I)$ for some $s \in (\frac 12, \min\,\{1,r\})$.
\end{proof}

In 2D the following approximation result holds (see \cite[Lemma~4.1]{BabuskaS_87_hpF}).

\begin{lemma} \label{lm_2D_p-approx}
Let $K$ be the reference triangle or square.
Then there exists a family of operators
$\{\pi_p\},\ p=1,2,\ldots,\ \pi_p:\, H^k(K)\rightarrow \CP_p(K)$
such that for any $f \in H^k(K)$, $k \ge 0$ there holds
\[
  \|f - \pi_p f\|_{H^s(K)} \le
  C p^{-(k-s)} \|f\|_{H^k(K)},\qquad  0\le s\le k.
\]
Moreover, $\pi_p$ preserves polynomials of degree $p$, i.e.,
$\pi_p f = f$ if $f \in \CP_p(K)$.
\end{lemma}

\subsection{The regularized Poincar{\' e} integral operators} \label{sec_Poincare}

In \cite{CostabelM_BRP}, Costabel and McIntosh studied a regularized version of the
Poincar{\' e}-type integral operator acting on differential forms in ${\field{R}}^n$.
They proved, in particular, that this operator is bounded on a wide range of
functional spaces including the whole scale of Sobolev spaces $H^r(\Omega)$
($r \in {\field{R}}$) on a bounded Lipschitz domain $\Omega$ which is starlike
with respect to an open ball. Moreover, the essential property of the classical
Poincar{\' e} map to preserve polynomials is retained by its regularized version.
Thus, the results of \cite{CostabelM_BRP} have immediate applications to the analysis
of high-order edge elements (see, e.g., \cite{Hiptmair_DCp,BespalovH_Chp} and the proof of
Theorem~\ref{thm_H(curl)_estimate} below).

Let us formulate some results of \cite{CostabelM_BRP} in two particular cases.
Namely, we will define two Poincar{\' e}-type integral operators: one operator
acts on scalar functions, and the other one acts on curl-free vector fields.
In both cases the functions and vector fields are defined on the reference
element (either triangle or square) $K$. Denoting by $B$ an open ball in $K$,
let us consider a smoothing function
\[
  \theta \in C^{\infty}({\field{R}}^2),\quad
  \supp\,\theta \subset B,\quad
  \int\limits_{B} \theta(\bolda)\,d\bolda = 1,\quad
  \bolda = (a_1,a_2).
\]
Then the first regularized Poincar{\' e}-type integral operator
$R:\,C^{\infty}(\bar K) \rightarrow (C^{\infty}(\bar K))^2$
(i.e., the operator acting on scalar functions) is defined as
$R\psi = (R_1,R_2)$, where
\beas
     R_1(\bx)
     & := &
     -\, \int\limits_{B} \theta(\bolda)\,(x_2 - a_2)
     \int\limits_0^1 t \psi(\bolda + t(\bx - \bolda))\,dt\,d\bolda,
     \\
     R_2(\bx)
     & := &
     \int\limits_{B} \theta(\bolda)\,(x_1 - a_1)
     \int\limits_0^1 t \psi(\bolda + t(\bx - \bolda))\,dt\,d\bolda.
\eeas
The second operator acting on vector fields is defined as follows:
\[
  \ba{l}
  A:\,(C^{\infty}(\bar K))^2 \rightarrow C^{\infty}(\bar K),
  \\
  \displaystyle{
  A\bu(\bx) :=
  \int\limits_{B} \theta(\bolda)\,
  \sum\limits_{i=1}^{2}
  (x_i - a_i) \int\limits_0^1 u_i(\bolda + t(\bx - \bolda))\,dt\,d\bolda,
  }
  \ea
\]
where $\bu = (u_1,u_2)$.

The following properties of the operators $R$ and $A$ are easy to check directly
(see also \cite[Proposition~4.2]{CostabelM_BRP}):

\begin{itemize}

\item[(R1)]
$R$ is a right inverse of the curl operator, i.e.,
\[
  \curl(R\psi) = \psi\qquad \forall\,\psi \in H^r(K),\quad r \ge 0;
\]

\item[(A1)]
if $\bu$ is curl-free, then $A$ is a right inverse of the gradient, i.e.,
\[
  \grad(A\bu) = \bu\qquad
  \forall\,\bu \in \bH^r(\curl 0, K) =
  \{\bu \in \bH^r(K);\; \curl\,\bu = 0\ \hbox{in $K$}\},\quad
  r \ge 0.
\]

\end{itemize}

Furthermore, the operators $R$ and $A$ satisfy the following continuity properties
(see \cite[Corollary~3.4]{CostabelM_BRP}):

\begin{itemize}

\item[(R2)]
the mapping $R$ defines a bounded operator $H^r(K) \rightarrow \bH^{r+1}(K)$
for any $r \ge 0$;

\item[(A2)]
the mapping $A$ defines a bounded operator $\bH^r(K) \rightarrow H^{r+1}(K)$
for any $r \ge 0$.

\end{itemize}

We will use the operators $R$ and $A$ to prove the following auxiliary lemma.

\begin{lemma} \label{lm_Poincare}
Let $\bu \in \bH^r(\curl,K)$, $r > 0$. Then there exist a function $\psi \in H^{r+1}(K)$
and a vector field $\bv \in \bH^{r+1}(K)$ such that
\be \label{P_1}
    \bu = \grad\,\psi + \bv.
\ee
Moreover,
\be \label{P_2}
    \|\bv\|_{\bH^{r+1}(K)} \le C\,\|\curl\,\bu\|_{H^r(K)}\qquad
    \hbox{and}\qquad
    \|\psi\|_{H^{r+1}(K)} \le C\,\|\bu\|_{\bH^r(K)}.
\ee
\end{lemma}

\begin{proof}
Since $\curl\,\bu \in H^r(K)$, we use the operator $R$ to define the vector field
$\bv := R(\curl\,\bu) \in \bH^{r+1}(K)$ (see property (R2)). Then
\be \label{P_3}
    \bu = (\bu - R(\curl\,\bu)) + R(\curl\,\bu) =
    (\bu - R(\curl\,\bu)) + \bv.
\ee
The vector field $(\bu - R(\curl\,\bu)) \in \bH^r(K)$ is curl-free due to property (R1).
Therefore, applying the operator $A$ to this vector field and using properties
(A1) and (A2), we find a function $\psi := A(\bu - R(\curl\,\bu)) \in H^{r+1}(K)$
such that $\grad\,\psi = \bu - R(\curl\,\bu)$. Therefore, the decomposition of $\bu$
in (\ref{P_3}) can be written in the form given by (\ref{P_1}).
The inequalities in (\ref{P_2}) are obtained by using the continuity properties
of the operators $R$ and $A$:
\[
  \|\bv\|_{\bH^{r+1}(K)} = \|R(\curl\,\bu)\|_{\bH^{r+1}(K)} \le
  C\,\|\curl\,\bu\|_{H^r(K)}
\]
and
\[
  \|\psi\|_{H^{r+1}(K)} = \|A(\bu - R(\curl\,\bu))\|_{H^{r+1}(K)}
  \le C\,\Big(\|\bu\|_{\bH^r(K)} + \|R(\curl\,\bu)\|_{\bH^{r}(K)}\Big)
  \le C\,\|\bu\|_{\bH^r(K)}.
\]
The last inequality relies on specific mapping properties of the scalar
curl operator (see, e.g.,~\cite{BespalovH_Chp}). This finishes the proof.
\end{proof}

\section{Interpolation operators} \label{sec_int}
\setcounter{equation}{0}

In \cite{DemkowiczB_03_pIE} two projection-based
interpolation operators have been introduced and analyzed. These are
the $H^1$-conforming interpolation operator
$\Pi^{1}_p: H^{1+r}(T) \rightarrow \CP_p(T)$ and
the $\bH(\curl)$-conform\-ing interpolation operator
$\Pi^{\curl}_p:\, \bH^r(T) \cap \bH(\curl,T) \rightarrow \bCP^{\rm Ned}_p(T)$
(here, $r > 0$ in both cases). Let us briefly sketch the definitions of both operators
and summarize their properties (see \cite{DemkowiczB_03_pIE} for details).

Let $g \in H^{1+r}(T)$, $r > 0$. To define the interpolant $\Pi^1_p\, g$, one starts
with the standard linear interpolation of $g$ at the vertices of $T$:
\[
  g_1 \in \CP_1(T),\quad g_1 = g \quad \hbox{at each vertex of $T$}.
\]
Then, for each edge $\ell \subset \partial T$, we define a polynomial
$g_{2,\ell}$ by using the projection
\be \label{int_1}
    g_{2,\ell} \in \CP_p^0(\ell):\quad
    \|(g - g_1)|_{\ell} - g_{2,\ell}\|_{\tH^{1/2}(\ell)} \rightarrow \min.
\ee
Extending $g_{2,\ell}$ by zero onto the remaining part of $\partial T$ (and keeping its
notation), using the polynomial extension $\CE_p$ from the boundary
(see Section~\ref{sec_ext}), and summing up over all edges we define
\be \label{int_2}
    g_2^p := \sum\limits_{\ell \subset \partial T} \CE_p(g_{2,\ell}) \in \CP_p(T).
\ee
Finally, we define the polynomial bubble $g_3^p$ by projection in the $H^1$-semi-norm
\be \label{int_3}
    g_{3}^p \in \CP_p^0(T):\quad
    |(g - g_1 - g_{2}^p) - g_3^p|_{H^1(T)} \rightarrow \min.
\ee
Then the interpolant $\Pi^1_p\, g$ is defined as the sum
\be \label{int_4}
    \Pi^1_p\, g := g_1 + g_2^p +g_3^p \in \CP_p(T).
\ee
Note that, due to the finite dimensionality of $g_1$, there holds for $r > 0$
\be \label{int_5}
    \|g_1\|_{H^{1+r}(T)} \simeq
    \sum\limits_{\nu:\,\hbox{\scriptsize vertices of $T$}} |g(\nu)| \le
    C\,\|g\|_{C(\bar T)} \le
    C\,\|g\|_{H^{1+r}(T)}.
\ee

Now we proceed to the $\bH(\curl)$-conforming interpolation operator.
Given a vector field $\bu \in \bH^r(T) \cap \bH(\curl,T)$ with $r > 0$,
the interpolant $\bu^p = \Pi^{\curl}_p\,\bu \in \bCP^{\rm Ned}_p(T)$
is also defined as the sum of three terms:
\[ 
    \bu^p = \bu_1 + \bu^p_2 + \bu^p_3.
\] 
Here, $\bu_1$ is the Witney (lowest order) interpolant
\[ 
    \bu_1 = \sum\limits_{\ell \subset \partial T}
    \Big(\int\limits_{\ell} \bn\times\bu\,d\sigma\Big)\,\bphi_{\ell},
\] 
where $\bn = (n_1,n_2)$ denotes the outward normal unit vector to $\partial T$,
$\bn\times\bu = n_1 u_2 - n_2 u_1$ with $\bu = (u_1,u_2)$, and
$\bphi_{\ell}$ are the standard basis functions (associated with edges $\ell$)
for $\bCP_1^{\rm Ned}(T)$.

For any edge $\ell \subset \partial T$ one has
\[ 
    \int\limits_{\ell} \bn \times (\bu - \bu_1)\,d\sigma = 0.
\] 
Hence, there exists a scalar function $\psi$, defined on the boundary $\partial T$,
such that
\[ 
    {\partial\psi\over{\partial \sigma}} =  \bn \times (\bu - \bu_1),\quad
    \psi = 0 \ \ \hbox{at all vertices}.
\] 
Then, for each edge $\ell$, the restriction $\psi|_{\ell}$ is projected in the
$\tH^{1/2}(\ell)$-norm onto the set of polynomials $\CP^0_{p+1}(\ell)$
\[ 
    \psi_{2,\ell} \in \CP^0_{p+1}(\ell):\quad
    \|\psi|_{\ell} - \psi_{2,\ell}\|_{\tH^{1/2}(\ell)} \rightarrow \min.
\] 
Extending $\psi_{2,\ell}$ by zero from $\ell$ onto $\partial T$ (and keeping
its notation) and using the polynomial extension $\CE_{p+1}$ from the boundary
we define $\psi_{2,p+1}^{\ell} := \CE_{p+1}(\psi_{2,\ell}) \in \CP_{p+1}(T)$.
Then we set
\[ 
    \bu_2^p = \sum\limits_{\ell\, \subset\, \partial T} \bu^p_{2,\ell} \in \bCP_p^{\rm Ned}(T),\ \ 
    \hbox{where \ } \bu^p_{2,\ell} = \grad\, \psi_{2,p+1}^{\ell}.
\] 
The interior interpolant $\bu^p_3$ is a vector bubble function that solves
the constrained minimization problem
\[
  \ba{l}
  \bu^p_3 \in \bCP^{\rm Ned,0}_p(T):
  \\[7pt]
  \|\curl(\bu - (\bu_1 + \bu_2^p + \bu_3^p))\|_{0,T} \rightarrow \min,
  \\[7pt]
  (\bu - (\bu_1 + \bu_2^p + \bu_3^p),\grad\,\phi)_{0,T} = 0\quad
  \forall \phi \in \CP^{0}_{p+1}(T).
  \ea
\]

These interpolation operators satisfy the following properties.

\begin{itemize}

\item[$1^{\circ}$.]
For $r > 0$ the operators $\Pi_p^{1}:\; \bH^{1+r}(T) \rightarrow H^{1}(T)$
and $\Pi_p^{\curl}:\; \bH^r(T) \cap \bH(\curl,T) \rightarrow \bH(\curl,T)$
are well defined and bounded, with corresponding operator norms independent
of the polynomial degree $p$ (cf. \cite[Propositions~1,~2]{DemkowiczB_03_pIE}).

\item[$2^{\circ}$.]
The operators $\Pi_p^{1}$ and $\Pi_p^{\curl}$ preserve scalar polynomials in
$\CP_p(T)$ and polynomial vector fields in $\bCP^{\rm Ned}_p(T)$, respectively.

\item[$3^{\circ}$.]
For $r > 0$ the following diagram commutes (see Proposition~3 in~\cite{DemkowiczB_03_pIE}):
\be \label{deRham}
    \ba{ccccc}
    H^{1+r}(T)                & \stackrel{\grad}{\longrightarrow}  &
    \bH^r(T) \cap \bH(\curl,T) & \stackrel{\curl}{\longrightarrow} & L^2(T)
    \cr
    \quad\left\downarrow{\Large\strut}\right.\,\Pi^1_{p+1}         &
                                                                   &
    \qquad\left\downarrow{\Large\strut}\right.\,\Pi^{\curl}_p      &
                                                                   &
    \qquad\left\downarrow{\Large\strut}\right.\,\Pi^{0}_{p-1}
    \cr
    \CP_{p+1}(T)          & \stackrel{\grad}{\longrightarrow} &
    \bCP^{\rm Ned}_p(T)   & \stackrel{\curl}{\longrightarrow} & \CP_{p-1}(T),
    \ea
\ee
where $\Pi^0_p:\, L^2(T) \rightarrow \CP_p(T)$ denotes the standard
$L^2$-projection onto the set of polynomials $\CP_p(T)$.

\end{itemize}

\section{Interpolation error estimates} \label{sec_estimate}
\setcounter{equation}{0}

First, we consider the $H^1$-conforming $p$-interpolation operator.
We prove that an optimal estimate can be obtained for the interpolation error
measured in the $H^1$-semi-norm.

\begin{theorem} \label{thm_H1_estimate}
Let $g \in H^{1+r}(T)$, $r >0$. Then there exists a positive constant $C$
independent of $p$ and $g$ such that
\be \label{H1_estimate}
    |g - \Pi_p^1\, g|_{H^1(T)} \le C\,p^{-r}\,\|g\|_{H^{1+r}(T)}.
\ee
\end{theorem}

\begin{proof}
Let $g \in H^{1+r}(T)$ with $r >0$.
Using the operator $\pi_p$ (see Lemma~\ref{lm_2D_p-approx}) and the polynomial extension
$\CE_p$ from the boundary, we define the following polynomial bubble function on $T$:
\[
  \varphi_p := \pi_p g - \CE_p (\g_{\rm tr}(\pi_p g)).
\]
Hereafter, $\g_{\rm tr}$ denotes the standard trace operator,
$\g_{\rm tr} f = f|_{\partial T}$. Then, making use of the definition of the interpolant
$\Pi_p^1\, g$ (see (\ref{int_2})--(\ref{int_4})), we have
\bea
    |g - \Pi_p^1\, g|_{H^1(T)}
    & \le &
    |(g - g_1 - g_2^p) - g_3^p|_{H^1(T)} \le
    |(g - g_1 - g_2^p) - \varphi_p|_{H^1(T)}
    \nonumber
    \\[5pt]
    & \le &
    |g - \pi_p g|_{H^1(T)} +
    \Big|g_1 + g_2^p - \CE_p (\g_{\rm tr}(\pi_p g))\Big|_{H^1(T)}
    \nonumber
    \\[5pt]
    & \le &
    \|g - \pi_p g\|_{H^1(T)} +
    \bigg|\CE_p \Big(\g_{\rm tr}(g_1) +
                     \sum\limits_{\ell \subset \partial T} g_{2,\ell} -
                     \g_{\rm tr}(\pi_p g)
                \Big)
    \bigg|_{H^1(T)}.
    \label{H1_proof_1}
\eea
Here we also used the fact that
$g_1 = \CE_1(\g_{\rm tr}(g_1)) = \CE_p(\g_{\rm tr}(g_1))$
for any $p \ge 1$ (see Remark~\ref{rem_ext} (ii)).
The extension operator $\CE_p$ satisfies (\ref{ext_2}). Therefore,
applying the continuity property of the trace
operator $\g_{\rm tr}:\;H^1(T) \rightarrow H^{1/2}(\partial T)$, we find
\beas
     \lefteqn{
              \bigg|\CE_p \Big(\g_{\rm tr}(g_1) +
                               \sum\limits_{\ell \subset \partial T} g_{2,\ell} -
                               \g_{\rm tr}(\pi_p g)
                          \Big)
              \bigg|_{H^1(T)}\qquad\qquad
             }
     \\[3pt]
     & &
     \qquad\qquad\qquad\le
     \Big\|
     \g_{\rm tr}(g_1) + \sum\limits_{\ell \subset \partial T} g_{2,\ell} - \g_{\rm tr}(\pi_p g)
     \Big\|_{H^{1/2}(\partial T)}
     \\
     & &
     \qquad\qquad\qquad\le
     C\,\bigg(
              \Big\|
              \g_{\rm tr}(g_1) + \sum\limits_{\ell \subset \partial T} g_{2,\ell} - \g_{\rm tr}(g)
              \Big\|_{H^{1/2}(\partial T)}
              +
              \Big\|
              \g_{\rm tr}(g - \pi_p g)
              \Big\|_{H^{1/2}(\partial T)}
     \bigg)
     \\
     & &
     \qquad\qquad\qquad\le
     C\,\bigg(
              \Big\|
              \sum\limits_{\ell \subset \partial T} g_{2,\ell} - \g_{\rm tr}(g - g_1)
              \Big\|_{H^{1/2}(\partial T)}
              +
              \|g - \pi_p g\|_{H^{1}(T)}
     \bigg).
\eeas
Using this estimate in (\ref{H1_proof_1}) we obtain
\be \label{H1_proof_2}
    |g - \Pi_p^1\, g|_{H^1(T)} \le
    C\,\|g - \pi_p g\|_{H^{1}(T)} +
    C\,\Big\|
       \sum\limits_{\ell \subset \partial T} g_{2,\ell} - \g_{\rm tr}(g - g_1)
       \Big\|_{H^{1/2}(\partial T)}.
\ee
The first term on the right-hand side of (\ref{H1_proof_2}) is estimated by applying
Lemma~\ref{lm_2D_p-approx}:
\be \label{H1_proof_3}
    \|g - \pi_p g\|_{H^{1}(T)} \le C\,p^{-r}\,\|g\|_{H^{1+r}(T)}.
\ee
To estimate the second term on the right-hand side of (\ref{H1_proof_2})
we use localization to the edges of the triangle, the definition of the edge
interpolants $g_{2,\ell}$ (see (\ref{int_1})), and the $p$-approximation result
in $\tH^{1/2}(\ell)$ on each edge (see Lemma~\ref{lm_1D_p-approx}).
One has
\beas
     \Big\|
     \sum\limits_{\ell \subset \partial T} g_{2,\ell} - \g_{\rm tr}(g - g_1)
     \Big\|_{H^{1/2}(\partial T)}
     & \le &
     C\,\sum\limits_{\ell \subset \partial T}
     \|g_{2,\ell} - (g - g_1)|_{\ell}\|_{\tH^{1/2}(\ell)}
     \\[5pt]
     & \le &
     C\,p^{-r}\,\sum\limits_{\ell \subset \partial T}
     \|(g - g_1)|_{\ell}\|_{H^{1/2+r}(\ell)}.
\eeas
Hence, applying the trace theorem for individual
edges (see \cite[Theorem~1.4.2]{Grisvard_92_SBV}) and estimating the norm
of $g_1$ as in (\ref{int_5}), we prove
\be \label{H1_proof_4}
   \Big\|
   \sum\limits_{\ell \subset \partial T} g_{2,\ell} - \g_{\rm tr}(g - g_1)
   \Big\|_{H^{1/2}(\partial T)} \le
   C\,p^{-r}\,\|g - g_1\|_{H^{1+r}(T)} \le
   C\,p^{-r}\,\|g\|_{H^{1+r}(T)}.
\ee
Now the desired error bound in (\ref{H1_estimate}) immediately
follows from (\ref{H1_proof_2})--(\ref{H1_proof_4}).
\end{proof}

In the following theorem we prove an optimal error estimate for
the $\bH(\curl)$-conforming $p$-interpolation operator $\Pi_p^{\curl}$.

\begin{theorem} \label{thm_H(curl)_estimate}
Let $\bu \in \bH^r(\curl, T)$, $r > 0$. Then there exists a positive constant $C$
independent of $p$ and $\bu$ such that
\be \label{H(curl)_estimate}
    \|\bu - \Pi_p^{\curl}\,\bu\|_{\bH(\curl,T)} \le
    C\, p^{-r}\, \|\bu\|_{\bH^r(\curl,T)}.
\ee
\end{theorem}

\begin{proof}
Given $\bu \in \bH^r(\curl, T)$ ($r > 0$), we use Lemma~\ref{lm_Poincare}
to decompose $\bu$ as follows:
\be \label{H(curl)_proof_1}
    \bu = \grad\,\psi + \bv,\qquad
    \psi \in H^{r+1}(T),\ \ \bv \in \bH^{r+1}(T).
\ee
Moreover (see (\ref{P_2})),
\be \label{H(curl)_proof_1_1}
    \|\bv\|_{\bH^{r+1}(T)} \le C\,\|\curl\,\bu\|_{H^r(T)}\qquad
    \|\psi\|_{H^{r+1}(T)} \le C\,\|\bu\|_{\bH^r(T)}.
\ee
Then, applying the interpolation operator $\Pi_p^{\curl}$ and using its commutativity
with $\Pi^1_{p+1}$ (see (\ref{deRham})), we write
\be \label{H(curl)_proof_2}
    \Pi_p^{\curl}\,\bu =
    \Pi_p^{\curl}(\grad\,\psi) + \Pi_p^{\curl}\,\bv =
    \grad(\Pi_{p+1}^1 \psi) + \Pi_p^{\curl}\,\bv.
\ee
Since $\Pi_p^{\curl}$ is a bounded operator preserving polynomials
(see properties~$1^{\circ}$,~$2^{\circ}$ of $\Pi_p^{\curl}$),
one has for some fixed $\eps \in (0,1)$
and for any polynomial $\bv_p \in (\CP_p(T))^2$:
\beas
     \|\bv - \Pi_p^{\curl}\,\bv\|_{\bH(\curl,T)}
     & = &
     \|\bv - \bv_p - \Pi_p^{\curl}(\bv - \bv_p)\|_{\bH(\curl,T)}
     \\[5pt]
     & \le &
     C\, \inf_{\bv_p \in (\CP_p(T))^2}
     \Big(
          \|\bv - \bv_p\|_{\bH^\eps(T)} + \|\curl(\bv - \bv_p)\|_{L^2(T)}
     \Big)
     \\[5pt]
     & \le &
     C\, \inf_{\bv_p \in (\CP_p(T))^2} \|\bv - \bv_p\|_{\bH^1(T)}.
\eeas
Hence, applying Lemma~\ref{lm_2D_p-approx} componentwise and using the first
inequality in (\ref{H(curl)_proof_1_1}), we estimate
\be \label{H(curl)_proof_3}
    \|\bv - \Pi_p^{\curl}\,\bv\|_{\bH(\curl,T)} \le
    C\,p^{-r}\,\|\bv\|_{\bH^{1+r}(T)} \le
    C\,p^{-r}\,\|\curl\,\bu\|_{H^{r}(T)}.
\ee
On the other hand, applying Theorem~\ref{thm_H1_estimate} and the second inequality
in (\ref{H(curl)_proof_1_1}) we obtain
\be \label{H(curl)_proof_4}
    \|\grad\,\psi - \grad(\Pi_{p+1}^{1}\psi)\|_{\bH(\curl,T)} =
    |\psi - \Pi_{p+1}^{1}\psi|_{H^1(T)} \le
    C\,p^{-r}\,\|\psi\|_{H^{1+r}(T)} \le
    C\,p^{-r}\,\|\bu\|_{\bH^{r}(T)}.
\ee
Combining (\ref{H(curl)_proof_3}) and (\ref{H(curl)_proof_4}) we prove
(\ref{H(curl)_estimate}) by making use of decompositions (\ref{H(curl)_proof_1}),
(\ref{H(curl)_proof_2}) and the triangle inequality.
\end{proof}

\section{Concluding remarks} \label{sec_remarks}
\setcounter{equation}{0}

The main result of the paper -- an optimal error estimate for 
the $\bH(\curl)$-conforming projection based $p$-interpolation operator
$\Pi_p^{\curl}$ introduced by Demkowicz and Babu{\v s}ka in~\cite{DemkowiczB_03_pIE} --
is proved for the second N{\' e}d{\' e}lec family of edge elements on the
reference triangle $T$ (see Theorem~\ref{thm_H(curl)_estimate}).
Below we discuss some simple extensions of this result.

{\em The first N{\' e}d{\' e}lec family of edge elements and the
     case of the reference square.}
Using appropriate polynomial spaces and the same constructions as in Section~\ref{sec_int},
one can define the $\bH(\curl)$-conforming $p$-interpolation operator
for the first N{\' e}d{\' e}lec family of edge elements introduced
in \cite{Nedelec_80_MFE}, and also for both
the first and the second families on the reference square $Q = (-1,1)^2$.
Moreover, properties $1^{\circ}$--$\;3^{\circ}$ of the operator $\Pi_p^{\curl}$
formulated in Section~\ref{sec_int} will remain valid in all these cases
(see \cite[Section~6]{DemkowiczB_03_pIE}). Thus, the proof of
Theorem~\ref{thm_H(curl)_estimate} carries over without modifications
to all the cases mentioned here.

{\em $H(\div)$-conforming $p$-interpolation operator.}
As mentioned in the introduction, the $\bH(\curl)$-conforming $p$-interpolation
operator is critical for the convergence and error analysis of the high-order
FEM with edge elements for Maxwell's equations in two dimensions.
However, when a boundary integral formulation of Maxwell's equations
(e.g., the electric field integral equation) is discretized by the
high-order boundary element methods, similar interpolation operators
are required in the $H(\div)$-conforming setting for RT- or BDM- surface
elements (see \cite{BespalovH_09_NpB,BespalovH_Chp}).
For the sake of completeness we will formulate here the main results related
to the operator $\Pi^{\curl}_p$ in the $\bH(\div)$-setting.
This can be easily done by rotation due to the isomorphism of the
curl and the $\div$ operators in 2D (and, as a consequence, the isomorphism
of the first (resp., second) N{\' e}d{\' e}lec family of edge elements and
the RT- (resp., BDM-) elements).

Let $K$ be either the equilateral reference triangle $T$ or the reference square $Q$.
We will focus the presentation on RT-elements only (in the case of BDM-elements
all formulations below are essentially the same).
Let $\bCP^{\rm RT}_p(K)$ be the RT-space of order $p\ge 1$ on
the reference element $K$ (see, e.g., \cite{BrezziF_91_MHF}), i.e.,
\[
  \bCP^{\rm RT}_p(K) =
  (\CP_{p-1}(K))^2 \oplus 
                         \Big(
                               \ba{c}
                               x_1 \\
                               x_2
                               \ea
                         \Big)
  \CP_{p-1}(K).
\]
Then the $\bH(\div)$-conforming $p$-interpolation operator 
$\Pi^{\div}_p:\, \bH^r(K) \cap \bH(\div,K) \rightarrow \bCP^{\rm RT}_p(K)$
($r > 0$) can be defined as the ``rotated'' $\bH(\curl)$-conforming interpolation
in the same way as in Section~\ref{sec_int}. It satisfies the following properties:

\begin{itemize}

\item[$1^{\circ}$.]
For $r > 0$ the operator
$\Pi_p^{\div}:\; \bH^r(K) \cap \bH(\div,K) \rightarrow \bH(\div,K)$
is bounded with corresponding operator norm being independent
of the polynomial degree $p$.

\item[$2^{\circ}$.]
The operator $\Pi_p^{\div}$ preserves polynomial vector fields in $\bCP^{\rm RT}_p(K)$.

\item[$3^{\circ}$.]
For $r > 0$ the following diagram commutes (cf. (\ref{deRham})):
\[
    \ba{ccccc}
    H^{1+r}(K)                & \stackrel{\bcurl}{\longrightarrow} &
    \bH^r(K) \cap \bH(\div,K) & \stackrel{\div}{\longrightarrow}   & L^2(K)
    \cr
    \quad\left\downarrow{\Large\strut}\right.\,\Pi^1_p         &
                                                               &
    \qquad\left\downarrow{\Large\strut}\right.\,\Pi^{\div}_p &
                                                               &
    \qquad\left\downarrow{\Large\strut}\right.\,\Pi^{0}_{p-1}
    \cr
    \CP_{p}(K)          & \stackrel{\bcurl}{\longrightarrow} &
    \bCP^{\rm RT}_p(K)  & \stackrel{\div}{\longrightarrow}   & \CP_{p-1}(K).
    \ea
\]
where $\bcurl = (\partial/\partial x_2,\,-\partial/\partial x_1)$.

\end{itemize}

The following theorem provides an optimal estimate for the interpolation error.

\begin{theorem} \label{thm_H(div)_estimate}
Let $\bu \in \bH^r(\div, K)$, $r > 0$. Then there exists a positive constant $C$
independent of $p$ and $\bu$ such that
\[
  \|\bu - \Pi_p^{\div}\,\bu\|_{\bH(\div,K)} \le
  C\, p^{-r}\, \|\bu\|_{\bH^r(\div,K)}.
\]
\end{theorem}

{\em $hp$-estimates.}
The projection-based interpolation operators $\Pi^{\curl}_p$ and $\Pi^{\div}_p$
preserve piecewise polynomial vector fields and provide conforming
approximations in $\bH(\curl)$ and $\bH(\div)$, respectively.
Therefore, using the standard Bramble-Hilbert argument and scaling,
Theorems~\ref{thm_H(curl)_estimate} and~\ref{thm_H(div)_estimate} extend
to the corresponding optimal $hp$-estimates on sequences of quasi-uniform
meshes of triangles and/or parallelograms satisfying the standard shape regularity
assumptions.

\bibliographystyle{siam}
\bibliography{bib,heuer,fem}

\end{document}